\documentclass[letterpaper,reqno,11pt,oneside,final]{amsart}
\pdfoutput=1

\usepackage{amsmath,amsthm,amsfonts,amssymb}
\usepackage[extension=pdf]{hyperref}
\usepackage{enumitem,comment}
\usepackage[T1]{fontenc}
\usepackage{times}
\usepackage{mathtools}
\usepackage[dvipsnames]{xcolor}
\usepackage{tikz}
\usepackage{mathrsfs}
\usepackage{cases}
\usepackage{bm}
\usepackage[multiple]{footmisc}
\usepackage{microtype}
\usepackage[totalheight=9.6in,totalwidth=6.15in,centering]{geometry}
\usepackage{graphics,graphicx}
\usepackage{tocvsec2}
\usepackage[color,notref,notcite]{showkeys}
\usepackage{xr}
\usepackage[normalem]{ulem}
\usepackage{upgreek}
\usepackage[style=alphabetic,sorting=nyt,natbib=true,maxnames=99,isbn=false,doi=false,url=false,firstinits=true,hyperref=auto,arxiv=abs,backend=bibtex]{biblatex}

\AtEveryBibitem{%
  \clearlist{language}}
\DeclareFieldFormat[article,inbook,incollection,inproceedings,patent,thesis,unpublished]{title}{#1\isdot}
\renewbibmacro{in:}{%
  \ifentrytype{article}{}{%
    \printtext{\bibstring{in}\intitlepunct}}} 
\defbibheading{apa}[\refname]{\section*{#1}}
\DeclareFieldFormat{sentencecase}{\MakeSentenceCase{#1}} \renewbibmacro*{title}{%
  \ifthenelse{\iffieldundef{title}\AND\iffieldundef{subtitle}}{}
  {\ifthenelse{\ifentrytype{article}\OR\ifentrytype{inbook}%
      \OR\ifentrytype{incollection}\OR\ifentrytype{inproceedings}%
      \OR\ifentrytype{inreference}} {\printtext[title]{%
        \printfield[sentencecase]{title}%
        \setunit{\subtitlepunct}%
        \printfield[sentencecase]{subtitle}}}%
    {\printtext[title]{%
        \printfield[titlecase]{title}%
        \setunit{\subtitlepunct}%
        \printfield[titlecase]{subtitle}}}%
    \newunit}%
  \printfield{titleaddon}}

\definecolor{darkblue}{rgb}{0.13,0.13,0.39}
\hypersetup{colorlinks=true,urlcolor=darkblue,citecolor=darkblue,linkcolor=darkblue,%
  pdftitle=RBMs,pdfauthor={M. Nica, J. Quastel, D. Remenik}}

\mathtoolsset{showmanualtags,showonlyrefs}

\def\dash---{\kern.16667em---\penalty\exhyphenpenalty\hskip.16667em\relax}

\let\oldmarginpar\marginpar
\renewcommand\marginpar[1]{\-\oldmarginpar[\raggedleft\footnotesize #1]%
  {\raggedright{\small\textsf{#1}}}}
\allowdisplaybreaks[2]
\numberwithin{equation}{section}

\newcommand{\fpref}[2]{\cite[#1 \ref*{fp-#2}]{fixedpt}}
\newcommand{\fpeqref}[1]{\cite[Eqn. (\ref*{fp-#1})]{fixedpt}}
\newcommand{\fprefst}[1]{\ref*{fp-#1}}

\newtheorem{thm}{Theorem}[section]

\newtheorem{prop}[thm]{Proposition}
\newtheorem{cor}[thm]{Corollary}
\theoremstyle{definition}
\newtheorem{rem}[thm]{Remark}
\newtheorem*{rem*}{Remark}

\newcommand{\I}{{\rm i}}
\newcommand{\pp}{\mathbb{P}}

\newcommand{\ee}{\mathbb{E}}
\newcommand{\rr}{\mathbb{R}}
\newcommand{\nn}{\mathbb{N}}
\newcommand{\zz}{\mathbb{Z}}

\newcommand{\p}{\partial}
\newcommand{\uno}[1]{\mathbf{1}_{#1}}
\newcommand{\ep}{\varepsilon}

\newcommand{\wt}{\widetilde}

\newcommand{\inv}[1]{\frac{1}{#1}}

\renewcommand{\d}{\mathrm{d}}

\newcommand{\ts}{\hspace{0.1em}}
\newcommand{\tts}{\hspace{0.05em}}
\newcommand{\tsm}{\hspace{-0.1em}}

\newcommand{\fh}{\mathfrak{h}}

\newcommand{\ff}{\mathfrak{f}}
\newcommand{\fX}{\mathfrak{X}}
\newcommand{\ft}{\mathbf{t}}
\newcommand{\fs}{\mathbf{s}}
\newcommand{\fx}{\mathbf{x}}
\newcommand{\fy}{\mathbf{y}}
\newcommand{\fz}{\mathbf{z}}

\newcommand{\fu}{\mathbf{u}}
\newcommand{\fa}{\mathbf{a}}
\newcommand{\fb}{\mathbf{b}}
\newcommand{\fB}{\mathbf{B}}

\newcommand{\aip}{\mathcal{A}}

\newcommand{\K}{K_{\Ai}}

\newcommand{\R}{R}

\newcommand{\SM}{\mathcal{S}}
\newcommand{\SN}{\bar{\mathcal{S}}}
\renewcommand{\P}{\chi}
\newcommand{\fK}{\mathbf{K}}
\newcommand{\fI}{\mathbf{I}}
\newcommand{\ftau}{\bm{\tau}}
\newcommand{\fT}{\mathbf{S}}

\newcommand{\g}{\bar{\bm{\gamma}}}

\newcommand{\Xb}{\mathbf{X}}
\newcommand{\Gb}{\mathbf{G}}

\newcommand{\Kb}[1][]{\fK^{#1{\uptext{RBM}}}}
\newcommand{\tKb}[1][]{\wt{\fK}^{#1{\uptext{RBM}}}}
\newcommand{\Xt}{X}
\newcommand{\Kt}{K^{\uptext{TSP}}}
\newcommand{\Kf}{\fK^{\uptext{FP}}}
\newcommand{\x}{x}
\newcommand{\Psit}{\Psi}
\newcommand{\Psib}{\bm{\Psi}}
\newcommand{\Phit}{\Phi}
\newcommand{\Phib}{\bm{\Phi}}
\newcommand{\h}{h}
\newcommand{\hb}{\mathbf{h}}

\newcommand{\bp}[1]{\bar{\p}^{(#1)}}
\newcommand{\lw}{l}

\makeatletter
\newcommand\RedeclareMathOperator{%
  \@ifstar{\def\rmo@s{m}\rmo@redeclare}{\def\rmo@s{o}\rmo@redeclare}%
}
\newcommand\rmo@redeclare[2]{%
  \begingroup \escapechar\m@ne\xdef\@gtempa{{\string#1}}\endgroup
  \expandafter\@ifundefined\@gtempa
     {\@latex@error{\noexpand#1undefined}\@ehc}%
     \relax
  \expandafter\rmo@declmathop\rmo@s{#1}{#2}}
\newcommand\rmo@declmathop[3]{%
  \DeclareRobustCommand{#2}{\qopname\newmcodes@#1{#3}}%
}
\@onlypreamble\RedeclareMathOperator
\makeatother
\newcommand{\uptext}[1]{\text{\upshape{#1}}}
\DeclareMathOperator{\epi}{\uptext{epi}}

\DeclareMathOperator{\UC}{\uptext{UC}}

\DeclareMathOperator{\Ai}{\uptext{Ai}}

\RedeclareMathOperator{\det}{\mathop{\uptext{det}}}
\RedeclareMathOperator{\ker}{\mathop{\uptext{ker}}}
\RedeclareMathOperator{\exp}{\mathop{\uptext{exp}}}
\RedeclareMathOperator{\log}{\mathop{\uptext{log}}}
\RedeclareMathOperator*{\lim}{\mathop{\uptext{lim}}}
\RedeclareMathOperator*{\sup}{\mathop{\uptext{sup}}}
\RedeclareMathOperator*{\limsup}{\mathop{\uptext{lim\hspace{1pt}sup}}}
\RedeclareMathOperator*{\max}{\mathop{\uptext{max}}}
\RedeclareMathOperator*{\inf}{\mathop{\uptext{inf}}}
\RedeclareMathOperator*{\min}{\mathop{\uptext{min}}}

\begin{document}

\title{One-sided reflected Brownian motions and the KPZ fixed point}

\date{October 21, 2020}

\author{Mihai Nica} \address[M.~Nica]{
  Department of Mathematics\\
  University of Toronto\\
  40 St. George Street\\
  Toronto, Ontario\\
  Canada M5S 2E4} \email{mnica@math.toronto.edu}

\author{Jeremy Quastel} \address[J.~Quastel]{
  Department of Mathematics\\
  University of Toronto\\
  40 St. George Street\\
  Toronto, Ontario\\
  Canada M5S 2E4} \email{quastel@math.toronto.edu}

\author{Daniel Remenik} \address[D.~Remenik]{
  Departamento de Ingenier\'ia Matem\'atica and Centro de Modelamiento Matem\'atico (UMI-CNRS 2807)\\
  Universidad de Chile\\
  Av. Beauchef 851, Torre Norte, Piso 5\\
  Santiago\\
  Chile} \email{dremenik@dim.uchile.cl}

\begin{abstract}
 We consider the system of one-sided reflected Brownian motions which is in variational duality with Brownian last passage percolation.
 We show that it has integrable transition probabilities, expressed in terms of Hermite polynomials and hitting times of exponential random walks, and that it converges in the 1:2:3 scaling limit to the KPZ fixed point, the scaling invariant Markov process defined in \cite{fixedpt} and believed to govern the long time large scale fluctuations for all models in the KPZ universality class.
 Brownian last passage percolation was shown recently in \cite{dov} to converge to the Airy sheet (or directed landscape), defined there as a strong limit of a functional of the Airy line ensemble.
 This establishes the variational formula for the KPZ fixed point in terms of the Airy sheet.
\end{abstract}

\maketitle

\section{RBM and the KPZ universality class}

The model of \emph{one-sided reflected Brownian motions} (RBM for short) is a system of reflected Brownian motions $\Xb_{t}(N)\le \Xb_t(N-1)\le \cdots \le \Xb_t(1)$ on $\rr$.
They start from an ordered initial condition $\Xb_{0}(N)< \Xb_0(n-1)<\dotsc <\Xb_0(1)$, perform Brownian motions, and interact with each other by one-sided reflections: $\Xb_t(k+1)$ is reflected to the left off $\Xb_t(k)$, so that the particles always remain ordered. 
This process and its variational dual, Brownian last passage percolation, have been studied intensively from many perspectives, see \cite{oconnellYor-Burke,oconnell-path-RSK,warren-dysonBM} and references therein for some early work, and \cite{gorinShkolnikov-multilevelTASEP,ferrariSpohnWeiss-book,dov,hammond-patchwork,aow-interldiff} for more recent results.
RBM can be defined formally in several equivalent ways.
The easiest is to start with $N$ standard Brownian motions $B_t(1) ,\dotsc,B_t(N)$ initially at $\Xb_0(1),\dotsc,\Xb_0(N)$, let $\Xb_t(1)=B_t(1)$, and then, recursively for $k=1,2,\ldots$, construct $\Xb_t(k)$ by reflecting $B_t(k+1)$ off $\Xb_t(k)$, the reflection $R_fB_t$ of a Brownian motion $B_t$ off any continuous function $f_t$ with $f_0\geq B_0$ being easy to define, e.g. by  the Skorokhod representation $R_fB_t= \min\{ \inf_{0\le s\le t} ( f_s +B_t-B_s ), B_t\}$.
Since the definition is recursive, it is not difficult to have $N= \infty$; the first $n<N$ particles don't even know the other $N-n$ are there.
Note also that using the Skorokhod representation one can let some of the initial positions coincide.
The system is alternatively defined by a system of stochastic differential equations involving the joint local times (see \cite[Sec. 4]{aow-interldiff} and references therein).  One can define other types of reflections, or point interactions for Brownian motions, but as far as is understood at this point, only this one has the integrability described in this article and we will not consider other models here.

From the Skorokhod representation it is not hard to see (see \cite[Eqn. (2.1.4)]{ferrariSpohnWeiss-book}, but note in that book the reflections go the opposite way) that RBM are in variational duality with \emph{Brownian last passage percolation} (BLPP): Given a family of independent, standard two-sided Brownian motions $(W_k)_{k\geq0}$ and $t\leq t'$ in $\rr$, $m\leq m'$ in $\zz$, if we define the \emph{last passage time}
\begin{equation}
\Gb[(t,m)\to(t',m')]=\sup_{t=t_m<\dotsm<t_{m'+1}=t'}\sum_{k=m}^{m'}\big(W_k(t_{k+1})-W_{k}(t_{k})\big)\label{eq:lpt}
\end{equation}
(i.e. the supremum over up-right paths from $(t,m)$ to $(t',m')$ along $\rr\times\zz$ of the Brownian increments collected along the path), then
\begin{equation}
\Xb_t(n)=\min_{\ell\leq n}\!\Big(\Xb_0(\ell)-\Gb[(0,\ell)\to(t,n)]\Big).\label{eq:RBMbLPP}
\end{equation}
The minimization here is over $\ell\ge 1$; equivalently we can add $\Xb_0(\ell)=\infty$, $\ell \le 0$, and minimize over all $\ell\leq n$.
Using the variational formula \eqref{eq:RBMbLPP}, one then extends to two-sided data: Given initial data $\big(\Xb_0(k)\big)_{k\in\zz}$, one defines the left hand side of \eqref{eq:RBMbLPP} for arbitrary $n\in\nn$ by the right hand side, without the restriction $\ell\geq 1$ in the minimum.
However, the minimum may not actually be attained if $\Xb_0(-\ell)$ grows too slowly as $\ell\to\infty$.  The correct condition comes from the growth rate of the last passage times $\Gb[(0,\ell)\to(t,n)]$, which is known to be of order $2\sqrt{(n-\ell)t}$ \cite{glynnWhitt,seppQueues}.
This can be used to show \cite[Prop. 2.4]{ferrariSpohnWeiss-book} that if, for example, $\Xb_0(-\ell)\geq c\tts\ell^{\frac12 + \delta} $ for all $\ell\geq0$ and some $\delta>0$, the minimum is attained for any $t>0$.

RBM and its variational dual BLPP are one of a small group of \emph{integrable} models which lie in the KPZ universality class, a huge class of one dimensional growth models and driven diffusive systems whose large scale fluctuations are conjectured to coincide with those of the Kardar-Parisi-Zhang equation.  
Models in the class can be interpreted as randomly growing height functions $h(t,x)$, $x\in \rr$, $t\ge 0$, which converge after zooming out, $\ep\to 0$ in the 1:2:3 scaling $h\longmapsto \fh_\ep(\ft,\fx) \coloneqq \ep^{1/2} h(\ep^{-3/2}\ft,\ep^{-1}\fx)- c_\ep t$, to a universal Markov process known as the KPZ fixed point, which can be described either through its integrable transition probabilities  \eqref{eq:onesideext2}, or alternatively by a variational formula \eqref{eq:var}/\eqref{eq:var1} involving a limiting multiparameter process called in various contexts the \emph{Airy sheet} or the \emph{directed landscape}.  So the KPZ fixed point and its variational dual the Airy sheet should be thought of as the universal limits in the class.

It is still a great challenge to understand the KPZ universality.
The limit for general initial data is only proved for one or two special models so far, using integrability in a very strong way.
In particular, the height function limit to the KPZ fixed point has only been proved for one model, the \emph{totally asymmetric simple exclusion process (TASEP)} \cite{fixedpt} (see \cite{mqr-variants} for an extension to several variants of TASEP).
On the other hand, the limiting Airy  sheet/directed landscape has only been proved for BLPP \cite{dov}.  Our purpose in this article is three-fold.
We show, 1.  RBM has integrable transition probabilities  (Thm. \ref{thm:RBM-formula}); 2. 
RBM converge to the KPZ fixed point in the 1:2:3 scaling (Thm. \ref{thm:RBMtoKPZ}) ; 3. the variational formula relating the KPZ fixed point and the Airy sheet  (Cor. \ref{cor:KPZfixedptAirySheet}) by combining 1 and 2 with results of \cite{dov}.  A second group \cite{dnv-scalinglongest} is working in obtaining the Airy sheet/directed landscape from exponential last passage percolation, which is in a similar duality with TASEP, providing an alternative route to 2 and 3.  These open the door to import methods  from one side  (e.g. \cite{hammond-patchwork}) to  the other.

RBM can actually be obtained as  the low density limit of TASEP.  Recall TASEP consists of particles on the lattice $\zz$ performing totally asymmetric nearest neighbour random walks with exclusion: Each particle independently attempts jumps to the neighbouring site to the right at rate $1$, the jump being allowed only if that site is unoccupied.
As for RBM we can consider initial conditions in which there is a rightmost particle; the particles remain ordered, and their positions can be denoted $\Xt_t(1)>\Xt_t(2)>\dotsm$.
Again, the dynamics of the first $N$ particles $\Xt_t(1)>\Xt_t(2)>\dotsm>\Xt_t(N)$ is independent of the rest, so the infinite system  makes sense.
Consider TASEP with $N$ particles started from initial positions $\Xt_{0}(1) > \Xt_{0}(2) \dotsc > \Xt_{0}(N)\in\zz$, let $\kappa>0$ be a scaling parameter, and choose initial particle positions so that $\Xt_{0}(i) = \sqrt{\kappa}\tts\Xb_{0}(i)$.
Then, in the sense of distributions \cite{gorinShkolnikov-multilevelTASEP}, 
\begin{equation}
\Xb_{t}(i) = \lim_{\kappa\to \infty} \kappa^{-1/2}(\Xt_{\kappa t}(i) -\kappa t).\label{eq:TASEPtoRBMs}
\end{equation}
This is used only for intuition in our proof; alternatively, if one added a few details (such as trace class convergence of the kernels), one could use our results to provide an alternative (but less elementary) proof of \eqref{eq:TASEPtoRBMs}.

In the next section we discuss the complete integrability of the transition probabilities of RBM, by which we mean a map from arbitrary one-sided initial data, to a formula for the $m$ point distributions at a later time (note that this excludes random initial data which we do not consider here.)  
Since RBM is a limit of TASEP, which was shown to have the same property in an earlier article \cite{fixedpt}, this is not such a surprise.
But the formulas now have a classical flavour from random matrix theory, being expressed naturally in terms of Hermite polynomials (instead of Charlier polynomials as for TASEP),
\begin{equation}\label{hermy}
H_{k}(x)=(-1)^ke^{\frac1{2}x^2}\p^ke^{-\frac1{2}x^2}.
\end{equation}

\section{Transition probabilities of RBM}

In \cite{ferrariSpohnWeiss-book}, the $m$ (spatial) point distributions of RBM and their asymptotics are computed for a few special initial conditions: Packed (converging to the Airy$_2$ process), periodic (or flat, Airy$_1$), stationary (Airy$_\uptext{stat}$), half-periodic (Airy$_{2\to1}$), half-Poisson (Airy$_{2\to\uptext{BM}}$), and periodic-Poisson (Airy$_{1\to\uptext{BM}}$).
Note also that, for special initial data, there has been a recent breakthrough in which two time distributions have been computed \cite{johansson-twotime-Br}.
We are now going to give a formula for the $m$ (spatial) point distributions of RBM for general right finite initial data, at a fixed later time $t$.
These generate the transition probabilities, in the same way that finite dimensional distributions define the Wiener measure.
 
Let $\p$ denote the derivative operator $\p f=f'$ and $\p^{-1}$ its formal inverse
\begin{equation}\label{eq:pinv}
\p^{-1}f(x)=\int_{-\infty}^x\d y\,f(y),
\end{equation}
which can be thought of as an integral operator with kernel $\p^{-1}(x,y)=\uno{x>y}$.
Also, for a fixed vector $a\in\rr^m$ and indices $n_1<\dotsc<n_m$ we introduce the operators
\begin{equation}\label{eq:defChis}
\P_a(n_j,x)=\uno{x>a_j},\qquad\bar\P_a(n_j,x)=\uno{x\leq a_j}
\end{equation}
(we will use the same notation if $a$ is a scalar, writing $\chi_a(x)=1-\bar\chi_a(x)=\uno{x>a}$).

\begin{thm}\label{thm:RBM-formula}
Consider RBM with initial condition $\{ \Xb_0(i) \}_{i=1}^{\infty}$. 
For any indices $1 \leq n_1 < n_2 < \dotsc <n_m$, any locations $a_1,\dotsc,a_m\in\rr$ and any $t > 0$, we have
\begin{equation}\label{eq:extKernelProbRBM}
\pp\big(\Xb_t(n_j) \geq a_j, j=1,\dotsc,m \big) = \det\!\left(\fI-\bar\chi_a\Kb_t\bar\chi_a\right)_{L^2(\{n_1,\dotsc,n_m\}\times\rr)},
\end{equation}
where $\det$ is the Fredholm determinant, with
\begin{multline}\label{eq:KtRBM-Hermite}
\Kb_t(n_i,z_i;n_j,z_j)=-\p^{-(n_j-n_i)}(z_i,z_j)\uno{n_i<n_j}\\
+\sum_{\ell=0}^{n_j-1}\iint_{\rr^2}\d\eta\,e^{\eta-b}\tts\pp_{B_0=\eta}(\tau=\ell,\,B_\tau\in\d b)\,\uppsi_{n_i}(t,\eta-z_i)\bar\uppsi_{n_j-\ell-1}(t,b-z_j),
\end{multline}
where $(B_k)_{k\geq0}$ is a discrete time random walk taking $\uptext{Exp}[1]$ steps to the left, $\tau$ is the hitting time of the epigraph of the curve $(\Xb_0(k+1))_{k\geq0}$ by $B_k$, i.e. $\tau=\inf\!\big\{k\geq0\!:B_k\geq\Xb_0(k+1)\big\}$, and
\begin{equation}\label{eq:uppsi}
\uppsi_n(t,x)=t^{-n/2}\ts \tfrac{1}{\sqrt{2\pi t}}e^{-\frac{x^2}{2t}} H_{n}(\tfrac{x}{\sqrt{t}}),\qquad\bar\uppsi_n(t,x)=\tfrac{1}{n!}t^{n/2}H_{n}(\tfrac{x}{\sqrt{t}}).
\end{equation}
\end{thm}

The proof is given in Sec. \ref{sec:TASEPtoRBM}.

The RBM kernel \eqref{eq:KtRBM-Hermite} can alternatively be written in a way to harmonize with \cite{fixedpt} here, and for later convenience, it is better to conjugate the kernel by multiplying by $e^{z_j-z_i}$ (which does not change the value of the Fredholm determinant):
\begin{equation}\label{eq:KtRBM-SNSM}
\begin{aligned}
\tKb_t(n_i,z_i;n_j,z_j)&\coloneqq e^{z_j-z_i}\Kb_t(n_i,z_i;n_j,z_j)\\
&=-e^{z_j-z_i}\p^{-(n_j-n_i)}(z_i,z_j)\uno{n_i<n_j}+(\SM_{-t,-n_i})^*\SN^{\epi(\Xb_0)}_{t,n_j}(z_i,z_j)
\end{aligned}
\end{equation}
with (see also \eqref{eq:SMcontour} and \eqref{eq:SNcontour} for contour integral formulas for these kernels)
\begin{equation}
\begin{aligned}
\SM_{-t,-n}(z_1,z_2)&=e^{z_1-z_2}\uppsi_{n}(t,z_1-z_2)=e^{z_1-z_2}(\p^n e^{\frac12t\p^2})^*(z_1,z_2),\\
\SN_{-t,n}(z_1,z_2)&=e^{z_2-z_1}\uppsi_{n-1}(t,z_1-z_2)=e^{z_2-z_1}\bar\p^{(-n)}e^{-\frac12t\p^2}(z_2,z_1),\\
\SN^{\epi(\Xb_0)}_{t,n}(z_1,z_2)&=\ee_{B_0=z_1}\!\left[\SN_{-t,n-\tau}(B_\tau,z_2)\uno{\tau<n}\right],
\end{aligned}\label{eq:SMSN-Basep}
\end{equation}
and where $\bar\p^{(-k)}(x,y)=\frac{(x-y)^k}{(k-1)!}$ is an ``analytic extension'' of $\p^{-k}$ (see \eqref{eq:defpinvext}, and also Rem. \ref{rem:bckwdsHK} for the meaning of $e^{-\frac12t\p^2}$ in this context).
Note that the kernel $e^{z_2-z_1}\p^{-1}(z_1,z_2)$ making up the first term on the right hand side of \eqref{eq:KtRBM-SNSM} is precisely the transition matrix of the exponential random walk $B_k$.
These equalities are also proved in Sec. \ref{sec:TASEPtoRBM}; the basic ingredient is the classical relation between Hermite polynomials and the heat kernel,
\begin{equation}
\p^ne^{\frac12t\p^2}(x,y)=(-1)^nt^{-n/2}\tfrac{1}{\sqrt{2\pi t}}\tts e^{-\frac1{2t}(x-y)^2}H_{n}(\tfrac{x-y}{\sqrt{t}}).\label{eq:Hmt1}
\end{equation}

Note that the formulas from \cite{ferrariSpohnWeiss-book} for special deterministic initial conditions can be recovered by computing explicitly the hitting law.
This is straightforward in the case of packed initial data: here $\Xb_0(k)=0$ for each $k\geq0$, so inside the expectation in \eqref{eq:SMSN-Basep} we have $\tau=0$ if $z_1\geq0$, in which case $B_\tau=z_1$, and otherwise $\tau=\infty$; using this together with \eqref{eq:SMcontour} and \eqref{eq:SNcontour} in \eqref{eq:KtRBM-SNSM} leads directly to their formula after removing the conjugation $e^{z_j-z_i}$.
When $m=1$ this is the classic Hermite kernel, and the distribution is the same as that of the largest eigenvalue of an $n\times n$ GUE random matrix \cite{gravnerTracyWidom,baryshnikov}, a remarkable fact establishing the key link between random matrix theory and random growth models, and providing much of the motivation for the study of this particular model.
For the half-periodic case $\Xb_0(k)=-k$, $k\geq0$ (which after a limit leads also to the full periodic initial condition) it turns out to be simpler to use the biorthogonal representation of the kernel, \eqref{eq:KtTASEP} below, together with Prop. \ref{prop:RBMbiorth}; see \fpref{Ex.}{example210} for the analogous computation in the case of TASEP.

\section{From RBM to the KPZ fixed point} 

The KPZ fixed point is a Markov process on the space $\UC$ of upper semi-continuous functions $\fh\!:\rr\longrightarrow\rr\cup\{-\infty\}$ satisfying $\fh(\fx)\le A|\fx| + B$ for some $A,B<\infty$, with the topology of local Hausdorff convergence of hypographs.
It is shown in \cite{fixedpt} that it is the limit of the 1:2:3 rescaled TASEP height functions, $\fh(\ft,\fx;\fh_0)=\lim_{\ep\to0}\ep^{1/2}\!\left[h_{2\ep^{-3/2}\ft}(2\ep^{-1}\fx) + \ep^{-3/2}\ft\right]$, as long as $\fh_0(\fx)=\lim_{\ep\to0} \ep^{1/2}h_{0}(2\ep^{-1}\fx)$, all in the sense of the topology of $\UC$, in probability; here we are using $\fh(\ft,\fx;\fh_0)$ to denote the state of the Markov process at time $\ft$ given initial state  $\fh_0$.
The KPZ fixed point is conjectured to be the universal limit under such scalings for models in the KPZ universality class (see \cite{quastelCDM,cqrFixedPt,fixedpt,mqrScaling} for more background on the KPZ fixed point).
Our next theorem proves this for RBM.

Before stating the result we present the precise definition of the KPZ fixed point through its transition probabilities (for simplicity we present here the formula which uses $\fK_{\lim}$ from \fpeqref{eq:Klim} rather than the one appearing in the main results of that paper).
For $\fx\in\rr$, $\ft>0$ let
\[\fT_{-\ft,\fx}(u)= \ft^{-1/3} e^{\frac{2 \fx^3}{3\ft^2}+\frac{z\fx}{\ft} }\tsm\Ai(\ft^{-1/3}u+\ft^{-4/3}\fx^2),\]
which is the integral kernel of the operator $e^{ \fx\partial^2-\frac{\ft}3\tts\partial^3}$.
For $\fh_0\in\UC$ with $\fh_0(\fx) = -\infty$ for $\fx>0$, let
\[\fT^{\epi(-\fh_0^-)}_{-\ft,\fx}(v,u)=\ee_{\fB(0)=v}\big[\fT_{-\ft,\fx-\ftau}(\fB(\ftau),u)\uno{\ftau<\infty}\big],\]
where $\ftau$ is now defined as the hitting time of the epigraph of $-\fh_0^{-}$ by $\fB(x)$, a Brownian motion with diffusion coefficient $2$, with $\fh_0^{-} (\fx) = \fh_0(-\fx)$. 
For such initial data, the KPZ fixed point kernel reads
\begin{equation}
\Kf_\ft(n_i,\cdot;n_j,\cdot)=-e^{(\fx_j-\fx_i)\p^2}\uno{\fx_i>\fx_j}+(\fT_{-\ft,\fx_i})^*\fT^{\epi(-\fh_0^-)}_{-\ft,-\fx_j},\label{fpp}
\end{equation}
and the transition probabilities for the KPZ fixed point $\fh(\ft,\fx)$ are defined through their finite dimensional distributions, which are given by Fredholm determinants,
\begin{equation}
\pp_{\fh_0}\!\left(\fh(\ft,\fx_1)\leq \fa_1,\dotsc,\fh(\ft,\fx_m)\leq \fa_m\right)=\det\!\left(\fI-\bar\P_{-\fa}\Kf_\ft\bar\P_{-\fa}\right)_{L^2(\{\fx_1,\dotsc,\fx_m\}\times\rr)}.\label{eq:onesideext2}
\end{equation}
The process is statistically spatially invariant, so the corresponding formula for $\fh_0\in\UC$ with $\fh_0(\fx) = -\infty$ for $\fx>\fx_0$ are easily recovered.
Such data are dense in $\UC$, and it is shown in \cite{fixedpt} the probabilities are continuous functions of $\fh_0\in \UC$.
So the general formula can be obtained from \eqref{eq:onesideext2} by approximation (see \fpref{Thm.}{prop:Kfixedptconv}).

Consider a family of initial conditions $\Xb_{0}^{(\ep)}$ for RBM satisfying, for some $\fh_0\in\UC$,
\begin{equation}\label{eq:X0ep}
- \ep^{1/2}(\Xb_{0}^{(\ep)}(-2\ep^{-1}\fx)-2\ep^{-1}\fx)\xrightarrow[\ep\to0]{}\fh_0(\fx)
\end{equation} 
in distribution in $\UC$, where the left hand side is interpreted as a linear interpolation to make it a continuous function of $\fx\in \rr$.
The left hand side of \eqref{eq:X0ep} (as well as the right hand side of \eqref{eq:Xbep} below) has the interpretation as a kind of inverse function of the height function (whose definition we have not given, and will not need, here; see (2.1) and Sec. 3 in \cite{fixedpt} for the definition in the context of TASEP).
The scaling limits look at perturbations of height functions from flat, so the limit of the inverse function naturally picks up a minus sign.
Note that the convergence \eqref{eq:X0ep} requires a far more restrictive lower bound on the growth of $\Xb_{0}(-\ell)$, $\ell>0$ than the one used after \eqref{eq:RBMbLPP} to show that the minimum is achieved; it is being assumed to grow linearly.
These conditions correspond to upper bounds needed on the initial data for the KPZ fixed point to prevent blowup:
Once the initial data grows quadratically a blowup occurs in finite time (see \cite{sarkarVirag} for the precise statement).
We assume linear upper bounds on $\UC$ because it is a nice class where the solution stays and exists for all time, and trace class bounds on the kernel turn out to be very tricky if one considers the general case of quadratic growth.

Let $\Xb^{(\ep)}_t$ denote RBM with this initial data and define the 1:2:3 rescaled RBM,
\begin{equation}\label{eq:Xbep}
\fX^{(\ep)}_\ft(\fx)=\ep^{1/2}\big(\Xb_{\ep^{-3/2}\ft}^{(\ep)}(\ep^{-3/2}\ft-2\ep^{-1}\fx)+2\ep^{-3/2}\ft-2\ep^{-1}\fx\big).
\end{equation}

\begin{thm}
\label{thm:RBMtoKPZ}
Assume that the initial data satisfies \eqref{eq:X0ep}.
Then, for each $\ft>0$, in $\UC$ in distribution, 
\begin{equation}
 -\fX^{(\ep)}_\ft(\fx)\xrightarrow[\ep\to0]{}\fh(\ft,\fx;\fh_0).\label{rbmconv}
\end{equation}
\end{thm}

By the Markov property the above convergence extends to multitime distributions.
One could also upgrade the result to space-time convergence by employing the variational formula and the Brownian Gibbs property \cite{corwinHammond-AiryLineEns} (or alternatively from the explicit formulas as done for TASEP in \cite{fixedpt}) to show that the trajectories are H\"older $\frac12-$ in space and H\"older $\frac13-$ in time and using this to derive the necessary tightness.

The rest of this section is devoted to the proof of Theorem \ref{thm:RBMtoKPZ}.
An important ingredient will be a variational formula for the limit of $\fX^{(\ep)}_\ft$ which follows from \cite{dov}; the formula is stated and proved in Section \ref{sec:RBMtosheet} (see Prop. \ref{prop:RBMtoAirySheet}).

At the level of convergence of kernels, the proof is analogous to the proof of the convergence of the TASEP kernels in \cite{fixedpt}, but in this case uses the standard convergence of Hermite polynomials to Airy functions.
Our goal is to study the limit of the kernel $\ep^{-1/2}\tKb_t(n_i,z_i;n_j,z_j)\uno{z_i\leq-\tilde a_i,\,u_j\leq-\tilde a_j}$
defined in terms of $\tKb_t$ from \eqref{eq:KtRBM-SNSM} with 
\begin{equation}\label{eq:scaling}
t=\ep^{-3/2}\ft,\quad n_i=\ep^{-3/2}\ft-2\ep^{-1}\fx_i,\quad z_i=-2\ep^{-3/2}\ft+2\ep^{-1}\fx_i+\ep^{-1/2}u_i,\quad\eta=\ep^{-1/2}v
\end{equation}
and $\tilde a_i=-2\ep^{-3/2}\ft+2\ep^{-1}\fx_i-\ep^{-1/2}\fa_i$ (the $\ep^{-1/2}$ in front of the kernel comes from the $z_i$ change of variables).
In view of \eqref{eq:onesideext2}, our goal is to show that this kernel converges in a suitable way to $\Kf_\ft(\fx_i,u_i;\fx_j,u_j)\uno{u_i\leq-\fa_i,\,u_j\leq-\fa_j}$.
Note that the change of variables transforms the indicator functions in the desired manner.

Recall first that $e^{y-x}\p^{-1}(x,y)$ is the transition probability for the exponential random walk $B_k$.
Thus, under this scaling and for $\fx_i>\fx_j$, the first term on the right hand side of \eqref{eq:KtRBM-SNSM} becomes $-\ep^{-1/2}$ times the probability density for the walk $B_k$ to go from $2\ep^{-1}\fx_i+\ep^{-1/2}u_i$ to $2\ep^{-1}\fx_j+\ep^{-1/2}u_j$ in time $2\ep^{-1}(\fx_i-\fx_j)$. This converges to $-e^{(\fx_i-\fx_j)\p^2}(u_i,u_j)$ by the Central Limit Theorem.

Define 
\begin{align}\label{essone}
\fT^\ep_{-\ft,\fx}(v,u)&= \ep^{-1/2}e^{-\frac12t}\SM_{-t,-n}(\eta,z)= \ep^{-1/2}e^{-\frac12t}e^{\eta-z}\uppsi_{n}(t,\eta-z),\\\label{esstwo}
\bar\fT^\ep_{-\ft,-\fx}(v,u)&=\ep^{-1/2}e^{\frac12t}\SN_{-t,n}(\eta,z)=\ep^{-1/2}e^{\frac12t}e^{z-\eta}\bar\uppsi_{n}(t,\eta-z)
\end{align}
so that, after scaling, the second term on the right hand side of \eqref{eq:KtRBM-SNSM} reads $(\fT^\ep_{-\ft,\fx_i})^*\bar\fT^{\ep,\epi(-\fh_0^-)}_{-\ft,-\fx_j}$ with $\bar\fT^{\ep,\epi(-\fh_0^-)}_{-\ft,-\fx_j}(v,u)=\ep^{-1/2}\ee_{\ep^{1/2}B_0=v}\Big[\bar\fT^{\ep}_{-\ft,-\fx_j-\inv2\ep\tau}(\ep^{1/2}B_\tau,u)\uno{\tau<n}\Big]$.
Now the standard asymptotics of Hermite polynomials (see e.g. \cite[Lem. 3.7.2]{AGZ})
\begin{equation}\label{eq:oscfns}
n^{1/12}\psi_n(2\sqrt{n}+n^{-1/6}x)\longrightarrow\Ai(x), \qquad \psi_{n}(x)=(2\pi)^{-1/4}(n!)^{-1/2}e^{-\frac{1}{4}x^{2}}H_{n}(x)
\end{equation}
gives the pointwise convergence 
\begin{equation}\label{f2}
\fT^\ep_{-\ft,\fx}(v,u)\longrightarrow\fT_{-\ft,\fx}(v,u),
\qquad\bar\fT^\ep_{-\ft,-\fx}(v,u)\longrightarrow\fT_{-\ft,-\fx}(v,u).
\end{equation}
Finally, using this scaling inside the integral in \eqref{eq:KtRBM-Hermite} (or in the expectation defining $\bar\fT^{\ep,\epi(-\fh_0^-)}_{-\ft,-\fx_j}$ above) and rescaling accordingly we get that the random walk $B$ converges to the Brownian motion $\fB$ and the hitting time $\tau$ converges to the Brownian hitting time $\ftau$ to the epigraph of curve $-\fh_0^-(\fy)$, leading to $\ep^{-1/2}\tKb_t(n_i,z_i;n_j,z_j)$ converging to the right hand side of \eqref{fpp}.

This proves the limit \eqref{rbmconv} at the level of pointwise convergence of the kernels involved.
The transition probabilities are given in terms of Fredholm determinants, so to prove they in fact converge, one needs more.
In order to provide the simplest possible proof, we take the following route. 
We show that if $\fh_0$ is made up of \emph{multiple narrow wedges} (see below) and $\Xb^{(\ep)}_0$ are natural approximations of such, the convergence of kernels is in trace norm.
Since the Fredholm determinant is continuous in the trace class topology, this proves the desired convergence of probabilities in the multiple narrow wedge case. 
Together with Prop. \ref{prop:RBMtoAirySheet}, this yields the variational formula 
\[\fh(\ft,\fx;\fh_0)\stackrel{\uptext{dist}}{=}\sup_{\fy\in\rr}\big\{\aip(0,\fx;\ft,\fy)-\tfrac{(\fx-\fy)^2}\ft+\fh_0(\fy)\big\}\]
for multiple narrow wedge $\fh_0$, where $\aip(\fs,\fx;\ft,\fy)$ is the Airy sheet constructed in \cite{dov} (see also \eqref{eq:var1}).
Multiple narrow wedges are dense in $\UC$ and both sides are continuous functions of $\fh_0\in \UC$, so we conclude the variational formula holds for all $\fh_0\in \UC$.
Using Prop. \ref{prop:RBMtoAirySheet} again we therefore conclude Thm. \ref{thm:RBMtoKPZ}.
(For more background on the Fredholm determinant, including the definition and properties of the Hilbert-Schmidt and trace norms to be used below, we refer to \cite{simon} or \cite[Sec. 2]{quastelRem-review}).

We start by defining approximate multiple narrow wedges.  
Fix $\fa_\ell<\dotsm<\fa_1\leq0$ and $\fb_1,\dotsc,\fb_\ell\in\rr$ and consider RBM initial data $\Xb^{(\ep)}_0(i)=\infty$ for $i<-2\ep^{-1}\fa_1$, $\Xb^{(\ep)}_0(i)=-2\ep^{-1}\fa_k-\ep^{-1/2}\fb_k$ for $2\ep^{-1}\fa_k\leq i<-2\ep^{-1}\fa_{k+1}$, $k=1,\dotsc,\ell-1$ and $\Xb^{(\ep)}(i)=-2\ep^{-1}\fa_\ell-\ep^{-1/2}\fb_\ell$ for $i\geq-2\ep^{-1}\fa_\ell$.
Then
\[-\fX^{(\ep)}_0(\fx)\xrightarrow[\ep\to0]{}\mathfrak{d}^{\fb_1,\dotsc,\fb_\ell}_{\fa_1,\dotsc,\fa_\ell}(\fx)\]
in $\UC$, where the \emph{multiple narrow wedge} $\mathfrak{d}^{\fb_1,\dotsc,\fb_\ell}_{\fa_1,\dotsc,\fa_\ell}(\fx)$ equals $\fb_k$ if $\fx$ is one of the $\fa_k$'s and $-\infty$ for all other $\fx$.

As explained above, Thm. \ref{thm:RBMtoKPZ} follows from 

\begin{prop}\label{prop:multinw}
Let $\fh_0=\mathfrak{d}^{\fb_1,\dotsc,\fb_\ell}_{\fa_1,\dotsc,\fa_\ell}$ and let $-\fX_0^{(\ep)}$ be their approximations prescribed above.
Then \eqref{rbmconv} holds.
\end{prop}

\begin{proof}
We need to prove the convergence of the scaled RBM kernel in trace norm.
For simplicity we will prove this only for the kernel corresponding to one-point distributions, i.e. the kernel $\tKb[(n),]_t$ coming from \eqref{eq:KtRBM-SNSM} (after scaling) with $n_i=n_j=n$; at the end of the proof we will explain how the proof extends to the general case.
Additionally, for notational simplicity we will take all $\fb_k$'s to be $0$; the extension to general $\fb_k$'s is straightforward, as will be clear from the proof.
We will write $\lw_k=\lceil-2\ep^{-1}\fa_k\rceil$.

Consider first the single narrow wedge case, $\ell=1$, for which $\Xb^{(\ep)}_0(i)=\infty$ for $i<-2\ep^{-1}\fa$ and $\Xb^{(\ep)}_0(i)=2\ep^{-1}\fa$ for $i\geq-2\ep^{-1}\fa$ (here $\fa=\fa_1$).
Since the walk $B_k$ takes strictly negative steps, $\tau<n$ if and only $n>\lw$ and $B_{\lw}>2\ep^{-1}\fa$, in which case $\tau=\lw=\lceil-2\ep^{-1}\fa\rceil$.
Then, recalling again that the transition matrix of the walk $B_k$ is $Q_\uptext{exp}(x,y)\coloneqq e^{y-x}\p^{-1}(x,y)$, we have $\SN_{-t,n}^{\epi(\Xb_0)}=(Q_\uptext{exp})^{\lw}\P_{2\ep^{-1}\fa}\SN_{-t,n-\lw}\uno{\lw<n}$, and thus (using $(\SM_{-t,-n})^*(Q_\uptext{exp})^\lw=(\SM_{-t,-n+\lw})^*$, which follows from \eqref{eq:SMSN-Basep})
\begin{equation}
\tKb[(n),]_t=(\SM_{-t,-n+\lw})^*\P_{2\ep^{-1}\fa}\SN_{-t,n-\lw_1}\uno{\lw_1<n}.\label{eq:tKb}
\end{equation}
From \eqref{eq:Hmt1} and the formula $e^{\frac1{2t}\p^2}\tsm(z)=\frac1{2\pi\I}\int_{\I\rr+\delta}dw\ts e^{tw^2/2-wz}$, valid for any $\delta\in\rr$, we get
\begin{equation}\label{eq:SMcontour}
\SM_{-t,-n}(z_1,z_2)=\frac{1}{2\pi\I}\int_{\I\rr+\delta}\d w\,w^n\tts e^{\frac12tw^2+(1-w)(z_1-z_2)},
\end{equation}
while using the representation $H_n(x)=\frac{n!}{2\pi\I}\int_{\gamma_0}\d w\,\frac{e^{-w^2/2+wx}}{w^{n+1}}$ we get
\begin{equation}\label{eq:SNcontour}
\SN_{-t,n}(z_1,z_2)=\frac{1}{2\pi\I}\int_{\gamma_0}\d w\,\frac{1}{w^{n+1}}e^{-\frac12tw^2+(w-1)(z_1-z_2)},
\end{equation}
where $\gamma_0$ is any positively oriented contour around the origin.
From \cite[Lemmas 5.12 and 5.13]{ferrariSpohnWeiss-book} we have, for each fixed $\ft,\fx$, that: 1.
The limits in \eqref{f2} hold uniformly on compact sets of $v,u$; 2. For each $\fa\in \rr$ there is a $C_\fa<\infty$
such that for $u>\fa$,
\begin{equation}\label{43}
|\fT^\ep_{-\ft,\fx}(u)|\le C_\fa\tts e^{-u}\qquad |\bar\fT^\ep_{-\ft,-\fx}(u)|\le C_\fa\tts e^{-u},
\end{equation}
where the operators $\fT^\ep_{-\ft,\fx}$, $\bar\fT^\ep_{-\ft,-\fx}$ from \eqref{f2} are written as functions of one variable given that they are convolution operators.
The scaled version of $\tKb[(n),]_t$ is $(\fT^\ep_{-\ft,\fx-\fa})^*\P_{0}\bar\fT^\ep_{-\ft,-\fx+\fa}$, which we need to consider acting on $L^2((-\infty,-\fa])$; we may split then the resulting operator as the product of two factors, $\bar\P_{-\fa}(\fT^\ep_{-\ft,\fx})^*\P_{0}$ and $\P_{0}\bar\fT^\ep_{-\ft,-\fx}\bar\P_{-\fa}$, whose Hilbert-Schmidt norms
\begin{equation}
\textstyle\|\bar\P_{-\fa}(\fT^\ep_{-\ft,\fx})^*\P_{0}\|_2^2=\int_{0}^\infty\!\d v \int_{-\infty}^\fa\!\d u\,\fT^\ep_{-\ft,\fx}(v,u)^2,~~\|\P_{0}\bar\fT^\ep_{-\ft,-\fx}\bar\P_{-\fa}\|_2^2=\int_{0}^\infty\!\d v\int_{-\infty}^\fa\!\d u\,\bar\fT^\ep_{-\ft,-\fx}(v,u)^2
\end{equation}\\[-20pt]
are bounded independent of $\ep$ by \eqref{43}.
Since the trace norm of a product is bounded by the product of the Hilbert-Schmidt norms, this provides a uniform (in $\ep$) bound on the trace norm of $(\fT^\ep_{-\ft,\fx})^*\P_{0}\bar\fT^\ep_{-\ft,-\fx}$ on $L^2((-\infty,-\fa])$. 
Now if we want to prove that $(\fT^\ep_{-\ft,\fx-\fa})^*\P_{0}\bar\fT^\ep_{-\ft,-\fx+\fa}$ converges to $(\fT_{-\ft,\fx-\fa})^*\P_{0}\bar\fT_{-\ft,-\fx+\fa}$ in trace norm in this space we can control the trace norm of the difference of the two sides by the sum of $\|\bar\P_{-\fa}(\fT^\ep_{-\ft,\fx})^*\P_{0}-\bar\P_{-\fa}(\fT_{-\ft,\fx})^*\P_{0}\|_2\|\P_{0}\bar\fT^\ep_{-\ft,-\fx}\bar\P_{-\fa}\|_2$ and $\|\bar\P_{-\fa}(\fT_{-\ft,\fx})^*\P_{0}\|_2 \|\P_{0}\bar\fT^\ep_{-\ft,-\fx}\bar\P_{-\fa}- \P_{0}\fT_{-\ft,-\fx}\bar\P_{-\fa}\|_2$, both of which vanish as $\ep\to 0$ by a simple truncation argument using the uniform convergence on compact sets together with \eqref{43} and the analog estimate for $|\fT_{-\ft,\fx}(u)|$ (see e.g. \fpeqref{eq:Stx-bd}).

Now we consider a general multiple narrow wedge.
The same argument as above shows that if the epigraph of $\Xb_0^{(\ep)}$ is hit, then it has to be hit at the beginning of one of the blocks of packed particles, and then by inclusion-exclusion one gets (recall $\lw_1<\dotsm<\lw_\ell$)
\begin{equation}
\begin{aligned}
\SN_{-t,n}^{\epi(\Xb_0)}
&=\sum_{k=1}^{\ell}(-1)^{k+1}\sum_{1\leq p_1\leq\dotsm\leq p_k\leq\ell}(Q_\uptext{exp})^{\lw_{p_1}}\P_{2\ep^{-1}\fa_{p_1}}(Q_\uptext{exp})^{\lw_{p_2}-\lw_{p_1}}\P_{2\ep^{-1}\fa_{p_2}}\dotsm\\
&\hspace{2.25in}\dotsm (Q_\uptext{exp})^{\lw_{p_k}-\lw_{p_{k-1}}}\P_{2\ep^{-1}\fa_{p_k}}\SN_{-t,n-\lw_{p_k}}\uno{\lw_{p_k}<n}.
\end{aligned}\label{eq:incl-excl}
\end{equation}
Therefore $\tKb[(n),]_t$ can be expressed as $\sum_{k=1}^{\ell}(-1)^{k+1}\sum_{1\leq p_1\leq\dotsm\leq p_k\leq\ell}K_{p_1,\dotsc,p_k}\ts\uno{\lw_{p_k}<n}$
with
\begin{equation}
K_{p_1,\dotsc,p_k}=(\SM_{-t,-n+\lw_{p_1}})^*\P_{2\ep^{-1}\fa_{p_1}}(Q_\uptext{exp})^{\lw_{p_2}-\lw_{p_1}}\P_{2\ep^{-1}\fa_{p_2}}\dotsm (Q_\uptext{exp})^{\lw_{p_k}-\lw_{p_{k-1}}}\P_{2\ep^{-1}\fa_{p_k}}\SN_{-t,n-\lw_{p_k}}.
\end{equation}
From the arguments used in the single narrow wedge case we know that $(\SM_{-t,-n+\lw_{p_1}})^*\P_{2\ep^{-1}\fa_{p_1}}$ and $\P_{2\ep^{-1}\fa_{p_k}}\SN_{-t,n-\lw_{p_k}}$ converge respectively, after properly scaling and truncating to $L^2((-\infty,-\fa])$, to $(\fT_{-\ft,\fx-\fa_{p_1}})^*\P_0$ and $\P_0\fT_{-\ft,-\fx+\fa_{p_k}}$, with the convergence holding in Hilbert-Schmidt norm.
The inner factor $\P_{2\ep^{-1}\fa_{p_1}}(Q_\uptext{exp})^{\lw_{p_2}-\lw_{p_1}}\P_{2\ep^{-1}\fa_{p_2}}\dotsm \P_{2\ep^{-1}\fa_{p_{k-1}}}(Q_\uptext{exp})^{\lw_{p_k}-\lw_{p_{k-1}}}\P_{2\ep^{-1}\fa_{p_k}}$, on the other hand, converges to $\P_{0}e^{(\fa_1-\fa_2)\p^2}\P_0\dotsm\P_0e^{(\fa_{k-1}-\fa_k)\p^2}\P_0$ under our scaling by the Central Limit Theorem, and it is not hard to see that the convergence holds in operator norm.
Recall that the trace norm satisfies the inequality $\|AK\|_1\leq\|A\|_\uptext{op}\|K\|_1$ with respect to the operator norm.
Thus using the same argument as for the single narrow wedge case, the whole product converges in trace norm to $(\fT_{-\ft,\fx-\fa_{p_1}})^*\P_{0}e^{(\fa_1-\fa_2)\p^2}\P_0\dotsm\P_0e^{(\fa_{k-1}-\fa_k)\p^2}\P_0\fT_{-\ft,-\fx+\fa_{p_k}}$.
Summing as on the right hand side of \eqref{eq:incl-excl} and using inclusion exclusion in the opposite direction then leads directly to $\Kf_\ft$ for $\fh_0=\mathfrak{d}^{0,\dotsc,0}_{a_1,\dotsc,\fa_\ell}$, since $\fB$ can only hit the epigraph of this function if it is positive at one of the points $\fa_1,\dotsc,\fa_\ell$.

We have proved the desired convergence of the kernel in \eqref{eq:KtRBM-SNSM} in the one-point case corresponding to $n_i=n_j=n$.
To prove the convergence in trace norm of the extended kernel it suffices to prove the convergence of  each of its entries (i.e. of the right hand side of \eqref{eq:KtRBM-SNSM} for general $n_i$ and $n_j$).
The kernel corresponding to $n_i,n_j$ has two terms.
The convergence of the second one follows from exactly the same arguments as above (the only difference is that in the decomposition \eqref{eq:tKb} the two pieces depend on different values of $n$, but this is not a problem since they where estimated separately anyway).
In order to upgrade to trace norm the convergence of the first term (for $n_i<n_j$), an additional conjugation by a multiplication operator is needed, but this conjugation does not affect the convergence of the other term; the argument is essentially the same as in \cite{fixedpt} (see Rem. \fprefst{rem:rwtoBMconjug} there), where the same conjugation is employed in the proof of the convergence of the TASEP kernels.
\end{proof}

\section{From RBM to the Airy sheet variational formula}\label{sec:RBMtosheet}

By coupling copies of TASEP starting with different initial conditions, and using compactness, the \emph{Airy sheet}
\begin{equation}
\aip(\fx,\fy)=\fh(1,\fy;\mathfrak{d}_\fx)+(\fx-\fy)^2 
\end{equation} 
is defined in \cite{fixedpt} as a two parameter process, where we have started the KPZ fixed point with the $\UC$ function $\mathfrak{d}_\fx(\fx) = 0$, $\mathfrak{d}_\fx(\fu) = -\infty$ for $\fu\neq\fx$; the {\em narrow wedge at $\fx$}.
Applying the preservation of maximum property of the KPZ fixed point, $\fh(\ft,\fx;\ff_1\vee\ff_2)\stackrel{\uptext{dist}}{=}\fh(\ft,\fx;\ff_1)\vee\fh(\ft,\fx;\ff_2)$, repeatedly to multiple narrow wedges and taking limits (see \cite[Sec. 4]{fixedpt}) leads to the variational formula
\begin{equation}\label{eq:var}
\fh(\ft,\fx;\fh_0)  \stackrel{\uptext{dist}}{=}  \sup_{\fy\in\rr}\big\{ \ft^{1/3}\aip(\ft^{-2/3} \fx,\ft^{-2/3} \fy)- \tfrac1{\ft}(\fx-\fy)^2 + \fh_0(\fy)\big\}.
\end{equation} 
The equality is in distribution, as functions of $\fx$, for fixed $\ft$.

More generally, one can consider the \emph{space-time Airy sheet} or \emph{directed landscape} \cite{cqrFixedPt,dov},
\begin{equation}
\aip(\fs,\fx,\ft,\fy)=  \fh(\ft,\fy;\fs, \mathfrak{d}_\fx)+ \tfrac{(\fx-\fy)^2 }{\ft-\fs},
\end{equation}
where $\fh(\ft,\fy;\fs, \bar\fh)$ denotes the KPZ fixed point at time $\ft$ starting at $\bar\fh$ at time $\fs$. 
This gives, for any $\fs\geq0$ and any $\bar\fh\in\UC$,
\begin{equation}\label{eq:var1}
\fh(\ft,\fx;\fs,\bar\fh)\stackrel{\uptext{dist}}{=}\sup_{\fy\in\rr}\big\{\aip(\fs,\fx;\ft,\fy)-\tfrac{(\fx-\fy)^2}{\ft-\fs}+\bar\fh(\fy)\big\},
\end{equation}
with equality in distribution in the space of continuous functions of $\ft\in [\fs, \infty)$ into $\UC$ (or, if one prefers,  in the space of continuous functions of $\ft\in (\fs, \infty)$ into H\"older $\tfrac12-$ functions.)
The variational formula \eqref{eq:var1} follows from \eqref{eq:var} by the Markov property and scaling invariance of the KPZ fixed point together with the semigroup property of the Airy sheet \cite[Thm. 4.18]{fixedpt}: If $\aip^1$ and $\aip^2$ are independent copies and $\ft_1+\ft_2=\ft$ are all positive, then for $\hat\aip^i(\fx,\fy)=\aip^i(\fx,\fy)-(\fx-\fy)^2$ one has $\sup_\fz\left\{ \ft_1^{1/3}\hat\aip^1(\ft_1^{-2/3} \fx,\ft_1^{-2/3} \fz) +  \ft_2^{1/3}\hat\aip^2(\ft_2^{-2/3} \fz,\ft_2^{-2/3} \fy) \right\} \stackrel{\uptext{dist}}{=} \ft^{1/3}\hat\aip^1(\ft^{-2/3} \fx,\ft^{-2/3} \fy)$ (this is called \emph{metric composition} in \cite{dov}).

The disadvantage of resorting to compactness arguments is that they cannot tell one that the limiting object is  defined uniquely and all one can conclude is that the variational expressions \eqref{eq:var} and \eqref{eq:var1} hold for any limit point.
In \cite{dov} this is overcome by showing that the sheet is a (non-explicit) functional of the Airy line ensemble \cite{corwinHammond-AiryLineEns}.
Since the Airy line ensemble is defined uniquely as a determinantal point process, the uniqueness of the Airy sheet follows as a consequence. 
We stress that, at this stage, nothing ensures that the Airy sheets coming from the two constructions in \cite{fixedpt} and \cite{dov} are the same object.
The equality, however, will be a consequence of our main result, Cor. \ref{cor:KPZfixedptAirySheet}.

The following result follows from the construction of the Airy sheet in \cite{dov} together with the relation between RBMs and Brownian last passage percolation.

\begin{prop}\label{prop:RBMtoAirySheet}
In the setting of Thm. \ref{thm:RBMtoKPZ}, assume now that $-\fX^{(\ep)}_0(\fx)\longrightarrow\fh_0(\fx)$ for some finitely supported $\fh_0\in\UC$, meaning that $\fh_0(\fx)=-\infty$ for $\fx$ outside some finite interval.
Then for each $\ft>0$,
\begin{equation}\label{eq:RMBtoAiryS}
-\fX^{(\ep)}_\ft(\fx)\xrightarrow[\ep\to0]{}\sup_{\fy\in\rr}\big\{\aip(0,\fx;\ft,\fy)-\tfrac{(\fx-\fy)^2}{\ft}+\fh_0(\fy)\big\}
\end{equation}
in distribution (locally) in $\UC$, where $\aip(\fs,\fx;\ft,\fy)$ stands for the the Airy sheet constructed in \cite{dov}.
\end{prop}

\begin{proof}
Thanks to the 1:2:3 scaling invariance of the KPZ fixed point, we may assume without loss of generality that $\ft=1$.
Additionally, by spatial invariance we may assume without loss of generality that the support of $\fh_0$ is contained in $(-\infty,0]$, so that we can work in the setting of RBM with a first particle $\Xb_t(1)$ as in the previous section.
Consider the scaling $t=\ep^{-3/2}$, $n=\ep^{-3/2}-2\ep^{-1}\fx$, so that $\fX^{(\ep)}_1(\fx)=\ep^{1/2}\big(\Xb^{(\ep)}_t(n)-t-n)$ and \eqref{eq:RBMbLPP} yield
\begin{align}
-\fX^{(\ep)}_1(\fx)&=\max_{1\leq\ell\leq n}\!\Big\{\ep^{1/2}\big(\!-\tsm\Xb^{(\ep)}_0(\ell)+\ell\big)+\ep^{1/2}\big(G[(0,\ell)\to(t,n)]+t+(n-\ell)\big)\Big\}.
\end{align}
Changing variables $\ell\mapsto-2\ep^{-1}\fy$, and in view of \eqref{eq:Xbep}, this becomes
\begin{multline}
\max_{\fy\in[\fx-\frac12\ep^{-1/2},\frac12\ep]}\!\Big\{-\tsm\fX^{(\ep)}_0(-2\fy)+\ep^{1/2}G[(0,-2\ep^{-1}\fy)\to(\ep^{-3/2},\ep^{-3/2}-2\ep^{-1}\fx)]\\
+2\ep^{-1}-2\ep^{-1/2}(\fx-\fy)\Big\}.
\end{multline}
$\fX^{(\ep)}_0(-2\fy)$ converges in distribution, locally in $\UC$, to $-\fh_0(\fy)$, and since $\fh_0(\fy)=-\infty$ for $\fy$ outside some finite interval $[\fa,\fb]$, then as $\ep\to 0$ we obtain the same thing by optimising over the set $[\fa-1,\fb+1]$.
Let $\fs=-2\ep^{1/2}\fy$ and $\fs'=1-2\ep^{1/2}\fx$, to get
\begin{multline}
\max_{\fy\in[\fa-1,\fb+1]}\!\Big\{\!-\tsm\fX^{(\ep)}_0(-2\fy)+\ep^{1/2}G[(\ep^{-3/2}\fs+2\ep^{-1}\fy,\ep^{-3/2}\fs)\to(\ep^{-3/2}\fs'+2\ep^{-1}\fx,\ep^{-3/2}\fs')]\\
+2\ep^{-1}-2\ep^{-1/2}(\fx-\fy)\Big\}.
\end{multline}
The variables $\fs,\fs',\fy$ vary over a compact set as $\ep\to0$, so assuming that $\fx$ also does, using the uniformity of the convergence in \cite[Thm. 1.5]{dov}, we deduce \eqref{eq:RMBtoAiryS} in distribution, locally in $\UC$.
\end{proof}

We can now fill in the gap between \cite{dov} and \cite{fixedpt} that results from the limiting objects in \cite{fixedpt} having been defined as limits from TASEP and those in \cite{dov} having been derived from BLPP.
Since the KPZ fixed point is defined through its transition probabilities, there is no ambiguity in its definition.
So Prop. \ref{prop:multinw} shows that the right hand side of \eqref{eq:RMBtoAiryS} is given by the KPZ fixed point $\fh(\ft,\fx;\fh_0)$ for multiple narrow wedge data (defined before Prop. \ref{prop:multinw}).
Since such data are dense in $\UC$, and both sides of \eqref{eq:var1} are continuous on $\UC$, we have that \eqref{eq:var1} holds for all $\fh_0\in \UC$.
At the same time, a second group \cite{dnv-scalinglongest} are filling the gap from the other side, proving that the result in \cite{dov} can also be obtained from exponential last passage percolation, which is in variational duality with TASEP.
Since the Airy sheet is obtained as the same functional of the Airy line ensemble, which is unique, again there is no ambigiuty, and TASEP converges to the unique Airy sheet of \cite{dov}. 
Either route leads to the main result:

\begin{cor}\label{cor:KPZfixedptAirySheet}
The KPZ fixed point Markov process constructed in \cite{fixedpt} and the (unique) Airy sheet/directed landscape constructed in \cite{dov} are related by the  variational formula \eqref{eq:var1}.
\end{cor}

\section{From TASEP to RBM}\label{sec:TASEPtoRBM}

\subsection{Biorthogonal ensemble for RBM}

Recall \eqref{eq:pinv} the operator $\p^{-1}f(x)=\int_{-\infty}^x\d y\,f(y)$, the notation being consistent with the fact that, when restricted to a suitable domain, $\p^{-1}$ is the inverse of the derivative operator $\p$.
Recall also that $\p^{-1}$ can be regarded as an integral operator with integral kernel $\p^{-1}(x,y)=\uno{x>y}$; more generally, it is easy to check that $\p^{-m}\coloneqq(\p^{-1})^m$ has integral kernel
\begin{equation}\label{eq:defQbm}
\p^{-m}(x,y)=\frac{(x-y)^{m-1}}{(m-1)!}\uno{x>y}.
\end{equation}

The next result can be derived by following the biorthogonalization approach introduced in \cite{sasamoto,borFerPrahSasam} for TASEP in the case of RBM (based on the explicit.
The proof is also contained implicitly in the proof of \cite[Prop. 4.2]{ferrariSpohnWeiss}, see also \cite[Lem. 3.5]{ferrariSpohnWeiss-book}.
Or it can alternately be derived by taking the low density limit of the result for TASEP from \cite{sasamoto,borFerPrahSasam}.

\begin{thm}\label{thm:BFPS-RBMs}
Consider RBM with  initial condition $\{ \Xb_0(i) \}_{i=1}^{\infty}$. 
For any indices $1 \leq n_1 < n_2 < \dotsc <n_m$, any locations $a_1,\dotsc,a_m\in\rr$ and any $t > 0$, 
\begin{equation}\label{eq:extKernelProbBFPS}
\pp\big(\Xb_t(n_j) > a_j, j=1,\dotsc,m \big) = \det\!\left(\fI-\bar\chi_a\Kb_t\bar\chi_a\right)_{L^2(\{n_1,\dotsc,n_m\}\times\rr)}
\end{equation}
with
\begin{equation}\label{eq:KtRBM}
\Kb_t(n_i,x_i;n_j,x_j)=-\p^{-(n_j-n_i)}(x_i,x_j)\uno{n_i<n_j}+\sum_{k=1}^{n_j}\Psib^{n_i}_{n_i-k}(x_i)\Phib^{n_j}_{n_j-k}(x_j),
\end{equation}
where
\begin{equation}\label{eq:defPsiRBM}
\Psib^n_k = \p^{k}e^{\frac12t\p^{2}}\delta_{\Xb_{0}(n-k)}
\end{equation}
for $k<n$ and the functions $\Phib^n_k$ are defined implicity by:
\begin{enumerate}[label={\normalfont (\arabic{*})}]
\item The biorthogonality relation $\langle\Psib_k^{n},\Phib_\ell^{n}\rangle_{L^2(\rr)}=\uno{k=\ell}$;
\label{orthoRBM}
\smallskip
\item $\Phib^n_k(x)$ is a polynomial of degree at most $n-1$ in $x$ for each $k$.\label{polyRBM}
\end{enumerate} 
\end{thm}

\begin{rem}
Using \eqref{eq:Hmt1} we have that $\Psib^n_k(x)=(-1)^kt^{k/2}h_{k}(t,x-X_0(n-k))w_t(x-X_0(n-k))$ with $h_k(t,\cdot)$ the scaled Hermite polynomials $h_k(t,x)=H_k(x/\sqrt{t})$, which are orthogonal with respect to the Gaussian weight $w_t(x)=(2\pi t)^{-1/2}e^{-\frac1{2t}x^2}$.
Hence (ignoring the prefactor $(-1)^kt^{k/2}$) the problem of finding the $\Phib^n_k$'s can be rephrased as follows:
\begin{quote}
For fixed $n>0$, and given a family of \emph{shifted Hermite functions} $(f_k)_{k=0,\dotsc,n-1}$, $f_k(x)=h_{k}(t,x-\Xb_0(n-k))w_t(x-\Xb_0(n-k))$, find a family of polynomials $(g_k)_{k=0,\dotsc,n-1}$, with $g_k$ of degree $k$, which is biorthogonal to $(f_k)_{k=0,\dotsc,n-1}$.
\end{quote}
\end{rem}

The challenge in such problems  is to actually find the biorthogonal functions $\Phib^n_k$ in a form which is useful.
Our strategy here is to compute them formally as a limit of the corresponding biorthogonal functions found for TASEP in \cite{fixedpt} and then simply check that the result satisfies \ref{orthoRBM} and \ref{polyRBM} above.

\subsection{TASEP}

$N$-particle TASEP was solved by \citet{MR1468391} using the coordinate Bethe ansatz, which leads to a formula for the transition probabilities given as the determinant of an explicit $N\times N$ matrix (an analogous formula can be written for RBM, see \cite[Prop. 8]{warren-dysonBM}, and is the starting point in the derivation of Thm. \ref{thm:BFPS-RBMs}).
However, and as in Thm. \ref{thm:BFPS-RBMs}, one is usually interested in studying $m$-point distributions of the process for arbitrary $m\leq N$, and moreover in obtaining formulas which are suitable for taking $N\to\infty$.
\cite{sasamoto,borFerPrahSasam} realized that this can be achieved by rewriting \citeauthor{MR1468391}'s formula in terms of a (signed) determinantal point process on a space of Gelfand-Tsetlin patterns and employing the Eynard-Mehta technology \cite{eynardMehta} to derive a Fredholm determinant formula for the $m$-point distributions.
The result is precisely the TASEP version of Thm. \ref{thm:BFPS-RBMs}, but we will not need to state it explicitly.
Instead, we will state a version of this result which follows from \cite{fixedpt}, where the biorthogonalization is performed explicitly.

In order to state the result we need to introduce some additional notation.  Define kernels
\begin{equation}\label{eq:defQ}
Q(x,y)=\uno{x>y},\qquad Q^{-1}(x,y)=\uno{x=y-1}-\uno{x=y}.
\end{equation}
They can be regarded as operators acting on suitable functions $f\!:\zz\longrightarrow\rr$, so for example
$Qf(x)=\textstyle\sum_{y<x}f(y)$, $Q^{-1}f(x)=\nabla^+f(x)\coloneqq f(x+1)-f(x)$.
Next we consider the kernel
\begin{equation}\label{eq:defPoissonKernel}
e^{-t\nabla^-}(x,y)=e^{-t}\frac{t^{x-y}}{(x-y)!}\uno{x\geq y}
\end{equation}
with $\nabla^-f(x)=f(x)-f(x-1)$.
$(e^{-t\nabla^-})_{t\geq0}$ is the semigroup of a Poisson process with jumps to the left at rate $1$; this formula is actually valid for all $t\in\rr$ and it defines the whole group of operators $(e^{-t\nabla^-})_{t\in\rr}$ (so, in particular $e^{-t\nabla^-}$ is invertible, with inverse $e^{t\nabla^-}$).

\begin{thm}[{\cite[Lem. 3.4]{borFerPrahSasam}, \cite[Thm. 2.2]{fixedpt}}]\label{thm:TASEP-biorth}
Suppose that TASEP starts with particles  $\Xt_0(1)>\Xt_0(2)>\dotsm$ and let $1\leq n_1<n_2<\dotsm<n_m$. 
Then, for $t>0$,
\begin{equation}\label{eq:extKernelProbTASEP}
  \pp\!\left(\Xt_t(n_j)>a_j,~j=1,\dotsc,m\right)=\det\!\left(\fI-\bar\chi_a\Kt_t\bar\chi_a\right)_{\ell^2(\{n_1,\dotsc,n_m\}\times\zz)}
\end{equation}
with
\begin{equation}\label{eq:KtTASEP}
\Kt_t(n_i,x_i;n_j,x_j)=-Q^{n_j-n_i}(x_i,x_j)\uno{n_i<n_j}+\sum_{k=1}^{n_j}\Psit^{n_i}_{n_i-k}(x_i)\Phit^{n_j}_{n_j-k}(x_j),
\end{equation}
where, for $k<n$,
\begin{equation}\label{eq:defPsiTASEP}
\Psit^{n}_k(x)=e^{-t\nabla^-}Q^{-k}\delta_{\Xt_0(n-k)}(x)=\frac1{2\pi\I}\oint_{\Gamma_0}\d w\,\frac{(1-w)^k}{w^{x+k+1-X_0(n-k)}}e^{t(w-1)}
\end{equation}
with $\Gamma_0$ any positively oriented simple loop including the pole at $w=0$ but not the one at $w=1$, and where the functions $\Phi_k^{n}(x)$, $k=0,\ldots,n-1$, are given by $\Phit^{n}_k(x)=(e^{t\nabla^-})^*\h^{n}_k(0,\cdot)(x)$ with $h_{k}^{n}(\ell,\cdot)\!:\zz\longrightarrow\rr$ defined recursively through $\h_{k}^{n}(k,z)  =1$, $z\in\zz$, and, for $\ell<k$, $\h_{k}^{n}(\ell,X_{0}(n-\ell))=0$, and
\begin{equation}\label{eq:defhTASEP}
\h_{k}^{n}(\ell,z)=\textstyle\sum_{y=x}^{X_{0}(n-\ell)}\h_{k}^{n}(\ell+1,y)
\end{equation}
where, as usual, $\sum_{y=x}^z h(y) = -\sum_{y=z}^x h(y)$ if $x>z$.
\end{thm}

\begin{rem}
We have stated the last result in a slightly different way than (but equivalent to) \cite{fixedpt}.
First, while in that paper the TASEP kernels were conjugated by $2^x$ in order to connect them directly to certain probabilistic objects (basically a random walk with Geom$[\frac12]$ steps), here we will omit that conjugation; this will allow us to state formulas in terms of slightly simpler kernels which are available in the continuous space setting of RBM.
Second, the TASEP biorthogonal functions $\Phit^n_k$, which solve a discrete space version of \ref{orthoRBM} and \ref{polyRBM} of Thm. \ref{thm:BFPS-RBMs}, were expressed in \cite{fixedpt} in terms of the solution of an initial--boundary value problem for a discrete backwards heat equation, while here we write down this solution explicitly; this allows us to compute the limiting $\Phib^n_k$'s very easily.
\end{rem}

\subsection{Brownian scaling limit of the TASEP biorthogonal functions}

We compute now the limit of the TASEP formulas under the scaling \eqref{eq:TASEPtoRBMs}.
The limits can be proved rigorously but since we don't need it, we will just state them and explain how they arise (see comments just after Thm. \ref{thm:BFPS-RBMs}.)

Since $Q$ is a discrete integration operator, it is not surprising that after scaling it converges to $\p^{-1}$.
In fact, we have for any $t\in\rr$ and $m\in\zz_{\geq0}$ that
\begin{equation}\label{eq:Qlim}
\kappa^{-(m-1)/2}Q^{m}(\sqrt{\kappa}\ts x+\kappa t,\sqrt{\kappa}\ts y+\kappa t)\xrightarrow[\kappa\to\infty]{}\p^{-m}(x,y),
\end{equation}
as can be checked for instance by using the explicit formula $Q^{m}(x,y)=\binom{x-y-1}{m-1}\uno{x\geq y+m}$.
The limit can be extended suitably to all $m\leq0$ to get that, after scaling, $Q^{-m}$ converges to $\p^{m}$.
Consider next the Poisson semigroup $(e^{-t\nabla^-})_{t\geq0}$.
We are interested in the scaling $\sqrt{\kappa}\ts e^{-\kappa t\nabla^-}\!(\sqrt{\kappa}\ts x,\sqrt{\kappa}\ts y+\kappa t)$, which is simply $\sqrt{\kappa}$ times the probability that a Poisson random variable with parameter $\kappa t$ equals $\sqrt{\kappa}\ts(y-x)+\kappa t$.
By the Central Limit Theorem,
\begin{equation}\label{eq:Poissonlim}
\sqrt{\kappa}\ts e^{-\kappa t\nabla^-}\!(\sqrt{\kappa}\ts x,\sqrt{\kappa}\ts y+\kappa t)\xrightarrow[\kappa\to\infty]{}e^{\frac12t\p^2}(x,y).
\end{equation}
Combining the above two facts leads directly to the following: For $k<n$, and replacing $t$ by $\kappa t$ and taking $\Xt_0(n-k)=\sqrt{\kappa}\ts\Xb_0(n-k)$ in \eqref{eq:defPsiTASEP}, 
\begin{equation}\label{eq:PsiTASEPlim}
\kappa^{k/2+1}\Psit^n_k(\sqrt{\kappa}\ts x+\kappa t)\xrightarrow[\kappa\to\infty]{}\p^{k}e^{\frac12t\p^2}\delta_{\Xb_0(n-k)}(x)=\Psib^n_k(x).
\end{equation}

We turn now to the $\Phit^n_k$'s.
For the functions $\h^n_k$ which are used to construct them (see \eqref{eq:defhTASEP}) we have
\begin{equation}\label{eq:hlim}
\kappa^{-(k-\ell)/2}\h^n_k(\ell,\sqrt{\kappa}\tts x)\xrightarrow[\kappa\to\infty]{}\hb^n_k(\ell,x)
\end{equation}
where $\hb^n_k(\ell,x)$ is given by $\hb_{k}^{n}(k,z)=1$ and 
\begin{equation}\label{eq:defhRBM}
\begin{split}
\hb_{k}^{n}(\ell,z)&=\int_x^{X_{0}(n-\ell)}\d y\,\hb_{k}^{n}(\ell+1,y) \quad\uptext{for}~\ell<k.
\end{split}
\end{equation}
Each $\hb^n_k(\ell,\cdot)$ is a polynomial of degree $k-\ell$.
We want to use \eqref{eq:hlim} to write a limit for $\Phit^n_k$.
This function is defined by applying $e^{t\nabla^-}$ to $h^n_k(0,\cdot)$, and from \eqref{eq:Poissonlim} we have, formally, that under the scaling we are interested in, $e^{t\nabla^-}$ should converge to $e^{-\frac12t\p^2}$.
Hence
\begin{equation}\label{eq:PhiTASEPlim}
\kappa^{-(k/2+1)}\Phit^n_k(\sqrt{\kappa}\ts x+\kappa t)\xrightarrow[\kappa\to\infty]{}e^{-\frac12t\p^2}\hb^n_k(0,\cdot).
\end{equation}

\begin{rem}\label{rem:bckwdsHK}
The backwards heat kernel appearing above does not make sense in general, but in this setting it is applied to the polynomial $\hb_{k}^{n}(0,\cdot)$, in which case its action can be defined by expanding it as a (finite) power series.
Furthermore, one can check that the group property $e^{s\p^2}e^{t\p^2}p=e^{(s+t)\p^2}p$ holds for any $s,t\in\rr$ and any polynomial $p$.
\end{rem}

The preceeding computations suggest the following result, which we prove directly.
It is worth noting how simple the proofs are using this method once one has the candidate biorthogonal functions. 

\begin{prop}\label{prop:RBMbiorth}
The functions $\Phib^n_k$ defined through \ref{orthoRBM} and \ref{polyRBM} of Thm. \ref{thm:BFPS-RBMs} are given explicitly by
\begin{equation}\label{eq:defPhib}
\Phib^n_k=e^{-\frac12t\p^2}\hb^n_k(0,\cdot).
\end{equation}
\end{prop}

\begin{proof}
The fact that $\Phib^n_k$ is a polynomial of degree at most $n-1$ follows from the above discussion and the same fact for $\hb^n_k(0,\cdot)$.
For the biorthogonality, we note first that by \eqref{eq:defhRBM} we have the simple identity $\p^{k}\mathbf{h}_{\ell}^{n}(0,X_{0}(n-k))=(-1)^{k}\uno{k=\ell}$.
Using this and the comment made in Rem. \ref{rem:bckwdsHK} we compute
\begin{align*}
\langle\Psib_k^{n},\Phib_\ell^{n}\rangle_{L^2(\rr)}&=\langle\p^{k}e^{\frac12t\p^{2}}\delta_{\Xb_{0}(n-k)},e^{-\frac12t\p^{2}}\hb^n_\ell(0,\cdot)\rangle_{L^2(\rr)}\\
&=\langle\p^{k}\delta_{\Xb_{0}(n-k)},\hb^n_\ell(0,\cdot)\rangle_{L^2(\rr)}=(-1)^k\p^k\hb^n_\ell(0,\Xb_0(n-k))=\uno{k=\ell}
\end{align*} 
as desired, where we have used the formula $\langle\delta^{(k)}_{x_0},f\rangle_{L^2(\rr)}=(-1)^kf^{(k)}(x_0)$ for $\delta^{(k)}_{x_0}$ the $k$-th distributional derivative of $\delta_{x_0}$.
\end{proof}

\subsection{Representation as hitting probabilities}

The next step is to represent $\Kb_t$ in terms of hitting times, thus producing a formula which is nicely set up for the 1:2:3 KPZ scaling limit.

Let $(B^*_k)_{k\geq0}$ denote a random walk with Exp$[1]$ steps to the right; its transition matrix is $Q_\uptext{exp}^*$ with $Q_\uptext{exp}(x,y)=e^{y-x}\p^{-1}(x,y)$ the transition matrix of the walk $B_k$ introduced earlier.
We claim that below the ``curve'' $(\Xb_0(n-\ell))_{\ell=0,\dotsc,n-1}$ defined by the initial data, $\hb_{k}^{n}(\ell,\cdot)$ can be represented as a hitting probability,
\begin{equation}\label{eq:hbnk-prob}
\hb_{k}^{n}(\ell,x)=e^{\Xb_{0}(n-k)-x}\ts\pp_{B_{\ell-1}^{\ast}=x}\!\left(\tau^{\ell,n}=k\right),\quad x<\Xb_0(n-\ell),
\end{equation}
with $\tau^{\ell,n}=\min\{\ell\leq k\leq n,B_{k}^{\ast}\geq\Xb_{0}(n-k)\}$.
This can be proved easily using \eqref{eq:defhRBM} and the formula
$\pp_{B_{\ell-1}^{\ast} = z}(\tau^{\ell,n}=k)=\int_{z}^{\Xb_{0}(n-\ell)}\d y\ts e^{-(y-z)}\pp_{B_{\ell}^{\ast}=y}(\tau^{\ell,n}=k)$, valid for $\ell<k$.
Next define
\begin{equation}\label{eq:defG0n}
G_{0,n}(x_{1},x_{2})=\sum_{k=0}^{n-1}\p^{-(n-k)}\delta_{\Xb_{0}(n-k)}(x_{1})\hb_{k}^{n}(0,x_{2}),
\end{equation}
so that 
\begin{equation}\label{eq:KTn-G0nned}
\Kb_t(n_i,\cdot;n_j,\cdot)=-\p^{-(n_j-n_i)}\uno{n_i<n_j}+\p^{n_i-n_j}e^{\frac12t\p^2}G_{0,n_j}e^{-\frac12t\p^2};
\end{equation}
note that the backwards heat kernel acts on the second variable of $G_{0,n_j}$, which defines a polynomial, so the action is well defined as in Rem. \ref{rem:bckwdsHK}.

Introduce the kernel
\begin{equation}\label{eq:defpinvext}
\bp{-m}(x,y)=\frac{(x-y)^{m-1}}{(m-1)!},
\end{equation}
which may be thought of as an analytic extension of $\p^{-m}(x,y)=\frac{(x-y)^{m-1}}{(m-1)!}\uno{x>y}$.
Note that $\big(\bp{-m}\big)_{m\geq0}$ is no longer a semigroup, but we do have $\p\bp{-m}=\bp{-m+1}$ for $m>1$ (however $\p\bp{-1}=0$).

\begin{prop}\label{prop:G0n}
Let $(B_{k})_{k\geq0}$ be a random walk taking \uptext{Exp}$[1]$ steps to the left and $\tau=\min\{k\geq0\!:B_{k}\geq\Xb_{0}(k+1)\}$.
For any $x_{1},x_{2}$,
\begin{equation}
G_{0,n}(x_{1},x_{2})=\ee_{B_{0}=x_{1}}\!\left[e^{x_1-B_{\tau}}\bp{-n+\tau}(B_{\tau},x_{2})\uno{\tau<n}\right].\label{eq:G0next}
\end{equation}
\end{prop}

\begin{proof}
We have $G_{0,n}(x_{1},x_{2})=\sum_{k=0}^{n-1}\pp_{B_{-1}^*=x_2}(\tau^{0,n}=k)\pp_{B^*_0=\Xb_0(n-k)}(B^*_{n-k}=x_1)e^{x_1-x_2}$ for $x_2<\Xb_0(n)$ by \eqref{eq:hbnk-prob} and the definition of $B^*_k$.
On the other hand, $\pp_{B^*_0=\Xb_0(n-k)}(B^*_{n-k}=x_1)=\int_{\Xb_0(n-k)}^\infty\d\eta\ts e^{\Xb_0(n-k)-\eta}\ts\pp_{B^*_k=\eta}(B^*_{n-1}=x_1)$ while, by the memoryless property of the exponential, $\pp_{B_{-1}^*=x_2}(\tau^{0,n}=k,B_k^*\in\d\eta)=\pp_{B_{-1}^*=x_2}(\tau^{0,n}=k)e^{\Xb_0(n-k)-\eta}\d\eta$ for $\eta\geq\Xb_0(n-k)$.
Thus
\[G_{0,n}(x_{1},x_{2})=e^{x_1-x_2}\ee_{B^*_{-1}=x_{2}}\!\left[(Q_\uptext{exp}^*)^{n-\tau^{0,n}}(B^*_{\tau^{0,n}},x_{1})\uno{\tau^{0,n}<n}\right]\]
Reversing the direction of the walk, the right hand side equals $e^{x_1-x_2}\ee_{B_{0}=x_{1}}\!\left[(Q_\uptext{exp})^{n-\tau}(B_\tau,x_{2})\uno{\tau<n}\right]=\ee_{B_{0}=x_{1}}\!\left[e^{x_1-B_\tau}\p^{-n+\tau}(B_\tau,x_{2})\uno{\tau<n}\right]$ with $\tau$ as in the statement.
This is valid for all $x_2<\Xb_0(n)$, and for those we have $B_\tau-x_2>\Xb_0(n-\tau)-\Xb_0(n)\geq0$, which implies that the formula we just obtained coincides with the right hand side of \eqref{eq:G0next} (for such $x_2$).
But by definition we know that $G_{0,n}(x_1,x_2)$ is a polynomial in $x_2$, and so is the right hand side of \eqref{eq:G0next}.
Since they coincide at infinitely many points, the result follows.
\end{proof}

Using
\begin{align}
\p^ne^{\frac12t\p^2}(x,y)&=(-1)^nt^{-n/2}\tfrac{1}{\sqrt{2\pi t}}\tts e^{-\frac1{2t}(x-y)^2}H_{n}(\tfrac{x-y}{\sqrt{t}})=\uppsi_n(y-x)&\uptext{($n\geq0$),\qquad}\\
\bp{-n}e^{-\frac1{2t}\p^2}(x,y)&=\tfrac{1}{(n-1)!}t^{(n-1)/2}H_{n-1}(\tfrac{x-y}{\sqrt{t}})=\bar\uppsi_{n-1}(x-y)&\uptext{($n\geq1$)\,\,\qquad}\label{eq:Hmt2}
\end{align}
(the first one is \eqref{eq:Hmt1} while for the second one we have used $e^{-\frac1{2t}\p^2}x^n=t^{n/2}H_n(\frac{x}{\sqrt{t}})$ and the fact that the parity of $H_n$ matches that of $n$ to turn $(-1)^nH_n(x)$ into $H_n(-x)$) we get from \eqref{eq:KTn-G0nned} that
\begin{multline}
\Kb_t(n,x_i;n,x_j)=-\p^{-(n_j-n_i)}(x_i,x_j)\uno{n_i<n_j}\\
+\int_{-\infty}^\infty\d\eta\,\uppsi_{n_i}(t,\eta-x_i)\ee_{B_0=\eta}\!\left[e^{\eta-B_{\tau}}\bar\uppsi_{n_j-\tau-1}(t,B_\tau-x_j)\uno{\tau<n_j}\right],\end{multline}
which yields Thm. \ref{thm:RBM-formula} as well as \eqref{eq:KtRBM-SNSM}.

\vskip20pt

\noindent{\bf Acknowledgements.}
MN and JQ were supported by the Natural Sciences and Engineering Research Council of Canada.
DR was supported by Conicyt Basal-CMM Proyecto/Grant PAI AFB-170001, by Programa Iniciativa Científica Milenio grant number NC120062 through Nucleus Millenium Stochastic Models of Complex and Disordered Systems, and by Fondecyt Grant 120194.

\printbibliography[heading=apa]

\end{document}